\documentclass[12pt]{amsart}
\usepackage{amsfonts}
\usepackage{amsfonts, mathdots}
\usepackage{amssymb,latexsym,color}
\usepackage{enumerate}\usepackage{arydshln}
\usepackage{hyperref}
\makeatletter
\@namedef{subjclassname@2010}{%
  \textup{2010} Mathematics Subject Classification}
\makeatother

\newtheorem{thm}{Theorem}[section]
\newtheorem{cor}[thm]{Corollary}
\newtheorem{lemma}[thm]{Lemma}

\theoremstyle{definition}

\newtheorem{rem}[thm]{Remark}
\newtheorem{eg}[thm]{Example}

\numberwithin{equation}{section}
\def\tr{{\rm tr}}

\begin{document}


\baselineskip=17pt



\title[determinantal inequalities]{Determinantal inequalities for block triangular matrices$^\diamondsuit$}

\author[M. Lin]{Minghua Lin}
\address{Department of Mathematics and Statistics\\
University of Victoria\\
 Victoria, BC, Canada, V8W 3R4.
}
\email{mlin87@ymail.com}

\date{}

\begin{abstract} Let  $T=\begin{bmatrix} X &Y\\ 0  &  Z\end{bmatrix}$ be an $n$-square matrix, where $X, Z$ are $r$-square and $(n-r)$-square, respectively. Among other determinantal inequalities,  it is proved
 \begin{eqnarray*}
\det\left(I_n+T^*T\right)\ge  \det\left(I_r+X^*X\right)\cdot \det\left(I_{n-r}+Z^*Z\right) \end{eqnarray*}
with equality holds if and only if $Y=0$.
\end{abstract}


\keywords{determinantal inequality, block  triangular matrices.  \\ $\diamondsuit$ Dedicated to Stephen Drury, whose beautiful works on matrix analysis inspire the author.
}

\maketitle

\section{Introduction}
The well known Fischer inequality  \cite[p. 506]{HJ13} states that if $A=\begin{bmatrix} A_{11} & A_{12}  \\ A_{12}^*  & A_{22}\end{bmatrix}$ is positive semidefinite, then \begin{eqnarray}\label{fis1} \det A\le \det A_{11}\cdot \det A_{22}.
\end{eqnarray}

As any positive semidefinite matrix $A$ can be factorized as $A=T^*T$ with $T=\begin{bmatrix} X &Y\\ 0  &  Z\end{bmatrix}$ being comformally partitioned as $A$, inequality (\ref{fis1}) can be written as \begin{eqnarray}\label{fis2} \qquad \det X^*X\cdot \det Z^*Z=\det T^*T \le \det X^*X\cdot \det (Y^*Y+Z^*Z).
\end{eqnarray}

This paper presents some results that complement (\ref{fis2}). We believe our results are of new pattern concerning determinantal inequalities.  Let us fix some notation.
The matrices considered here have entries from the field of complex numbers. $X', \overline{X}, X^*$ stands for transpose, (entrywise)conjugate, conjugate transpose of $X$, respectively.  For two $n$-square Hermitian matrices $X, Y$, we write $X\ge Y$ to mean $X-Y$ is positive semidefinite (so $X\ge 0$ means $X$ is positive semidefinite). The $n$-square identity matrix is denoted by $I_n$. The Frobenius norm of $X$ is denoted by $\|X\|_F$. It is known that $\|X\|_F=\sqrt{\tr X^*X}$, where $\tr$ denotes the trace.    If  $X=\begin{bmatrix} X_{11} & X_{12}  \\ X_{21}  & X_{22}\end{bmatrix}$ is an $n$-square matrix with $X_{11}$ nonsingular, then the Schur complement of $X_{11}$ in $X$ is defined by $X/X_{11}=X_{22}-X_{21}^*X_{11}^{-1}X_{12}$. A well known property of the Schur complement is  $\det X=\det X_{11}\cdot \det(X/X_{11})$.   Finally, for an $n$-square matrix, we denote by $\lambda_j(X)$ and $\sigma_j(X)$, $j=1, \ldots, n$, the eigenvalues and singular values  of $X$, respectively, such that $|\lambda_1(X)|\ge \cdots \ge |\lambda_n(X)|$ and $\sigma_1(X)\ge \cdots \ge \sigma_n(X)$.

\section{Main results}
We present the following result, showing that when more matrices are summed, the identity in (\ref{fis2}) becomes an inequality.

\begin{thm}\label{thm1} Let $T_k=\begin{bmatrix} X_k &Y_k\\ 0  &  Z_k\end{bmatrix}$, $k=1, \ldots, m$, be $n$-square conformally partitioned  matrices. Then
 \begin{eqnarray} \label{thm1e}
\det\left(\sum_{k=1}^m T_k^*T_k\right)\ge\det\left(\sum_{k=1}^m X_k^*X_k\right)\cdot\det\left(\sum_{k=1}^m Z_k^*Z_k\right).\end{eqnarray}    \end{thm}
\begin{proof} By a standard continuity argument, we may assume $X_k^*X_k$ is nonsingular for $k=1, \ldots, m$. As
\begin{eqnarray*} \begin{bmatrix} X_k^*X_k & X_k^*Y_k  \\ Y_k^*X_k   & Y_k^*Y_k \end{bmatrix}=\begin{bmatrix}X_k&Y_k\end{bmatrix}^*\begin{bmatrix}X_k&Y_k\end{bmatrix}\ge 0,
\end{eqnarray*} summing for $k$ from $1$ to $m$ gives
 \begin{eqnarray*} \begin{bmatrix} \sum_{k=1}^mX_k^*X_k & \sum_{k=1}^mX_k^*Y_k  \\ \sum_{k=1}^mY_k^*X_k   & \sum_{k=1}^mY_k^*Y_k \end{bmatrix}\ge 0,
\end{eqnarray*}  Hence, \begin{eqnarray*}  \sum_{k=1}^mY_k^*Y_k -\left(\sum_{k=1}^mY_k^*X_k\right)\left(\sum_{k=1}^mX_k^*X_k\right)^{-1} \left(   \sum_{k=1}^mX_k^*Y_k\right)\ge 0.
\end{eqnarray*}
On the other hand, $T_k^*T_k=\begin{bmatrix} X_k^*X_k &   X_k^*Y_k  \\  Y_k^*X_k   &   Y_k^*Y_k+Z_k^*Z_k \end{bmatrix}$. Thus
\begin{eqnarray*} \left(\sum_{k=1}^mT_k^*T_k\right)\Big/\left(\sum_{k=1}^mX_k^*X_k\right)&=&\sum_{k=1}^m(Y_k^*Y_k+Z_k^*Z_k)\\&&-\left(\sum_{k=1}^mY_k^*X_k\right)\left(\sum_{k=1}^mX_k^*X_k\right)^{-1} \left(\sum_{k=1}^mX_k^*Y_k\right)\\&\ge& \sum_{k=1}^mZ_k^*Z_k\ge 0.
\end{eqnarray*}
Applying the determinant on both sides gives the desired inequality.\end{proof}

\begin{rem} By a simple induction, Theorem \ref{thm1} can be extended to the $p\times p$ ($p>2$) block upper triangular case.   \end{rem}

A full characterization of the equality case in (\ref{thm1e}) is nasty, so we do not include it here. We extract a special case of Theorem \ref{thm1} with $m=2$ for later use. Moreover, the equality case is concise.

\begin{cor} \label{c0}   Let  $T=\begin{bmatrix} X &Y\\ 0  &  Z\end{bmatrix}$ be an $n$-square matrix, where $X, Z$ are $r$-square and $(n-r)$-square, respectively. Then
 \begin{eqnarray}\label{c0e} \det\left(I_n+T^*T\right)\ge  \det\left(I_r+X^*X\right)\cdot \det\left(I_{n-r}+Z^*Z\right).
 \end{eqnarray} Equality holds in (\ref{c0e}) if and only if $Y=0$.  \end{cor}
\begin{proof} It suffices to show the equality case. If $Y=0$, clearly (\ref{c0e}) becomes an equality. Conversely, if the equality in  (\ref{c0e}) holds, then
\begin{eqnarray}\label{c0e1} \begin{aligned} &\det\Big(I_{n-r}+Z^*Z+Y^*Y-Y^*X(I_r+X^*X)^{-1}X^*Y\Big)\\ =&\det(I_{n-r}+Z^*Z).\end{aligned}
\end{eqnarray} As $Y^*Y-Y^*X(I_r+X^*X)^{-1}X^*Y\ge 0$, the equality (\ref{c0e1}) holds   only if $Y^*Y-Y^*X(I_r+X^*X)^{-1}X^*Y=0$, i.e,
\begin{eqnarray*}Y^*\Big(I_r-X(I_r+X^*X)^{-1}X^*\Big)Y=0.
\end{eqnarray*} Now for any $j=1, \ldots, r$, \begin{eqnarray*} \lambda_j\Big(I_r-X(I_r+X^*X)^{-1}X^*\Big)&=&1-\frac{\lambda_{r-j+1}(X^*X)}{1+\lambda_{r-j+1}(X^*X)}\\&=&\frac{1}{1+\lambda_{r-j+1}(X^*X)}>0.
\end{eqnarray*} so $I_r-X(I_r+X^*X)^{-1}X^*$ is positive definite, forcing $Y=0$.   \end{proof}

\begin{cor} \label{c1} Let $T_k=\begin{bmatrix} X_k &Y_k\\ 0  &  Z_k\end{bmatrix}$, $k=1, \ldots, m$, be $n$-square conformally partitioned  matrices. If $X_k, Z_k$ are all normal matrices, then
 \begin{eqnarray}\label{c1e}
\det\left(\sum_{k=1}^m T_k^*T_k\right)\ge \left|\det\left(\sum_{k=1}^m \overline{X}_kX_k\right)\right|\cdot\left|\det\left(\sum_{k=1}^m \overline{Z}_kZ_k\right)\right|.\end{eqnarray} \end{cor}
\begin{proof} In view of Theorem \ref{thm1}, it suffices to show \begin{eqnarray*}\det\left(\sum_{k=1}^m X_k^*X_k\right)\ge \left|\det\left(\sum_{k=1}^m \overline{X}_kX_k\right)\right|.\end{eqnarray*} As \begin{eqnarray*} \begin{bmatrix} \sum_{k=1}^m\overline{X}_kX_k' & \sum_{k=1}^m\overline{X}_kX_k  \\ \sum_{k=1}^m X_k^*X_k'   & \sum_{k=1}^mX_k^*X_k \end{bmatrix}=\sum_{k=1}^m \begin{bmatrix} X_k'& X_k\end{bmatrix}^*\begin{bmatrix} X_k'& X_k\end{bmatrix}\ge 0,
\end{eqnarray*} this yields \begin{eqnarray*} \begin{bmatrix} \det\left(\sum_{k=1}^m\overline{X}_kX_k'\right) & \det \left(\sum_{k=1}^m\overline{X}_kX_k\right) \\ \det\left(\sum_{k=1}^m X_k^*X_k'\right)   &  \det\left(\sum_{k=1}^mX_k^*X_k\right) \end{bmatrix} \ge 0,
\end{eqnarray*} and so \begin{eqnarray*} \det\left(\sum_{k=1}^m\overline{X}_kX_k'\right)\cdot\det\left(\sum_{k=1}^mX_k^*X_k\right)\ge \left|\det\left(\sum_{k=1}^m \overline{X}_kX_k\right)\right|^2.
\end{eqnarray*} But $X_k$, $k=1, \ldots, m$, are normal, and so \begin{eqnarray*} \det\left(\sum_{k=1}^m\overline{X}_kX_k'\right)&=&\det\left(\sum_{k=1}^m\overline{X}_kX_k'\right)'\\ &=&\det\left(\sum_{k=1}^mX_kX_k^*\right) =\det\left(\sum_{k=1}^mX_k^*X_k\right),\end{eqnarray*}  as desired.  \end{proof}

 The author does not know if there is a simple  characterization for the equality case in (\ref{c1e}). The following example shows that (\ref{c1e}) may fail without the normality assumption.

\begin{eg} Taking
\begin{eqnarray*}T_1=\left[\begin{array}{cc:cc} -9&10&5&12\\-7&10&-11&-10 \\ \hdashline
 0&0&-2&3 \\ 0&0&2&26 \end{array}\right], \qquad   T_2=\left[\begin{array}{cc:cc} 13&-16&3&3\\-7& 9&3&11\\ \hdashline
 0&0&3&-16 \\ 0&0&-7&-13 \end{array}\right],
\end{eqnarray*} a simple calculation using Matlab gives \begin{eqnarray*}\det\left(T_1^*T_1+T_2^*T_2\right)&=&1.25 \times 10^8\\&<&\left|\det\left(X_1^2+X_2^2\right)\right|\cdot \left|\det\left(Z_1^2+Z_2^2\right)\right|=9.93\times 10^8.\end{eqnarray*} \end{eg}

Our next result says that, to some extent, Corollary \ref{c1} can be improved.
\begin{thm}\label{thm2}  Let $T=\begin{bmatrix} X &Y\\ 0  &  Z\end{bmatrix}$ be an $n$-square matrix, where $X, Z$ are $r$-square and $(n-r)$-square, respectively. Then
 \begin{eqnarray}\label{thm2e}
\det\left(I_n+T^*T\right)\ge  \det\left(I_r+\overline{X}X\right)\cdot \det\left(I_{n-r}+\overline{Z}Z\right).\end{eqnarray}
Equality holds in (\ref{thm2e})  if and only if $Y=0$ and $X, Z$ are symmetric.
   \end{thm}

\begin{rem}   Compared with (\ref{c1e}), we don't use absolute value on the right hand side of (\ref{thm2e}). This is because $\det\left(I_r+\overline{X}X\right)\ge 0$, an observation by   Djokovi\'{c}  \cite{Djo76, DL76}. However, it is possible that $\det\left(\sum_{k=1}^m \overline{X}_kX_k\right)<0$  in (\ref{c1e}). For example, taking \begin{eqnarray*} X_1=X_2'=\begin{bmatrix} 1 &   2  \\ 0  & 1\end{bmatrix},
\end{eqnarray*}  a simple calculation gives  \begin{eqnarray*} \det(\overline{X}_1X_1+\overline{X}_2X_2)=\det \begin{bmatrix} 2 &  4 \\4  & 2\end{bmatrix}=-12<0.
\end{eqnarray*}\end{rem}

 We need a lemma, which plays a key role in establishing the equality case in Theorem \ref{thm2}.

 \begin{lemma}\label{lem1} Let $X$ be an $n$-square matrix. Then
 \begin{eqnarray}\label{lem1e}
  \det(I_n+X^*X)\ge \det\left(I_n+\overline{X}X\right).\end{eqnarray}
 Equality holds in (\ref{lem1e})  if and only if $X$ is symmetric.    \end{lemma}
 \begin{proof}    From the  proof of Corollary \ref{c1}, we have
 $$\det (I_n+XX^*)\cdot\det(I_n+X^*X)\ge \det\left(I_n+\overline{X}X\right)^2.$$
 Note that $$\det(I_n+\overline{X}X')=\det(I_n+\overline{X}X')'=\det(I_n+XX^*)=\det(I_n+X^*X).$$ This proves (\ref{lem1e}).

 If $X$ is symmetric, then $\overline{X}=X^*$ and so  $\det\left(I_n+\overline{X}X\right)=\det(I_n+X^*X)$.

 We  show the other way around.  It is clear that
   $$\det\left(I_n+\overline{X}X\right)=\prod_{j=1}^n\Big(1+\lambda_j(\overline{X}X)\Big)\le \prod_{j=1}^n\Big(1+|\lambda_j(\overline{X}X)|\Big)$$ with the second inequality becoming an equality only if $\lambda_j(\overline{X}X)\ge 0$ for all $j$.

   By Weyl's inequality \cite[p. 317]{MOA11}, for $k=1, \ldots, n$,  \begin{eqnarray*} \prod_{j=1}^k|\lambda_j(\overline{X}X)|&\le& \prod_{j=1}^k \sigma_j(\overline{X}X)\\&\le& \prod_{j=1}^k  \sigma_j(\overline{X})\sigma_j(X)=\prod_{j=1}^k  \sigma_j^2(X)=\prod_{j=1}^k  \sigma_j(X^*X),
  \end{eqnarray*}
   where equality holds when $k=n$.   The strict convexity of the function $f(t)=\log(1+e^{t})$ (\cite[p. 156]{MOA11}) implies that
   $$\prod_{j=1}^n\Big(1+|\lambda_j(\overline{X}X)|\Big)\le \prod_{j=1}^n\Big(1+\sigma_j(X^*X)\Big)=\det(I_n+X^*X)$$
   with  the first inequality becoming an equality only if $\overline{X}X$ is normal.

    Thus, if $\det\left(I_n+\overline{X}X\right)=\det(I_n+X^*X)$, then $\overline{X}X\ge 0$ and $\lambda_j(\overline{X}X)=\sigma_j(X^*X)$ for all $j$. In particular, $\tr \overline{X}X=\tr X^*X$.  We shall show the latter implies that $X$ is symmetric.
     Compute
  \begin{eqnarray*} \|X-X'\|_F^2&=& \tr (X-X')^*(X-X')\\&=& \tr X^*X- \tr \overline{X}X -\tr (\overline{X}X)^*+\tr \overline{X}X'
  \\&=&2\Big(\tr X^*X-\tr \overline{X}X\Big)=0,
\end{eqnarray*}
  and so $X=X'$, as required.     \end{proof}

\noindent {\it Proof of  Theorem \ref{thm2}.}  The inequality (\ref{thm2e}) follows from (\ref{c0e}) and (\ref{lem1e}).

 If $Y=0$ and $X, Z$ are symmetric, then   \begin{eqnarray*}\det\left(I_n+T^*T\right)&=&\det(I_r+X^*X)\cdot\det(I_{n-r}+Z^*Z) \\&=&\det\left(I_r+\overline{X}X\right)\cdot \det\left(I_{n-r}+\overline{Z}Z\right).
\end{eqnarray*}

Conversely, if  the equality holds in (\ref{thm2e}), then actually we have  \begin{eqnarray*}\det\left(I_n+T^*T\right)&=&\det(I_r+X^*X)\cdot\det(I_{n-r}+Z^*Z) \\&=&\det\left(I_r+\overline{X}X\right)\cdot \det\left(I_{n-r}+\overline{Z}Z\right).
\end{eqnarray*} In view of Corollary \ref{c0}, the first equality gives $Y=0$. By Lemma \ref{lem1}, the second equality implies that $X, Z$ are symmetric. \qed

 An immediate consequence of  Theorem \ref{thm2} is the following, which is due to Drury  \cite[Lemma 4]{Dru02}.
 \begin{cor} {\bf (Drury's inequality)}  Let $T=[t_{ij}]$ be an $n$-square upper triangular matrix.  Then
 \begin{eqnarray*}
\det\left(I_n+T^*T\right)\ge \prod_{j=1}^n(1+|t_{jj}|^2).\end{eqnarray*} Equality holds if and only if $T$ is diagonal.  \end{cor}

\section{More results}

The absolute value of a  matrix $X$ is defined to be the matrix $|X|=(X^*X)^{1/2}$, the unique positive semidefinite square root of $X^*X$. The  inequality    (\ref{thm1e}) can be rewritten as
\begin{eqnarray}\label{e31} \det\left(\sum_{k=1}^m |T_k|^2\right)\ge\det\left(\sum_{k=1}^m |X_k|^2\right)\cdot\det\left(\sum_{k=1}^m |Z_k|^2\right).
\end{eqnarray}

The following result is an extension of Corollary \ref{c0}.
\begin{thm}\label{thm3}  Let $T=\begin{bmatrix} X &Y\\ 0  &  Z\end{bmatrix}$ be an $n$-square matrix, where $X, Z$ are $r$-square and $(n-r)$-square, respectively. Then for any $p\ge 1$
 \begin{eqnarray} \label{thm3e}
\det\left(I_n+|T|^p\right)\ge  \det\left(I_r+|X|^p\right)\cdot \det\left(I_{n-r}+|Z|^p\right).\end{eqnarray}
Equality holds in (\ref{thm3e}) if and only if $Y=0$.  \end{thm}
 \begin{proof} The proof is similar to the one given in \cite[Lemma 4]{Dru02}. Let $X=U_1|X|$, $Z=U_2|Z|$ be the polar decomposition of $X, Z$, respectively. Taking $U=\begin{bmatrix} U_1 & 0 \\ 0&U_2\end{bmatrix}$ (so $U$ is unitary) gives $U^*A= \begin{bmatrix} |X| &U_1^*Y\\ 0  &  |Z|\end{bmatrix}$. We have by Weyl's inequality \cite[p. 317]{MOA11}
 \begin{eqnarray*} \prod_{j=1}^k|\lambda_j(U^*T)|\le \prod_{j=1}^k \sigma_j(U^*T)=\prod_{j=1}^k \sigma_j(T), \qquad k=1, \ldots, n,
  \end{eqnarray*}
 i.e.,
 \begin{eqnarray*} \prod_{j=1}^k\lambda_j\left(\begin{bmatrix} |X| &0\\ 0  &  |Z|\end{bmatrix}\right)\le \prod_{j=1}^k \sigma_j(T), \qquad k=1, \ldots, n,
\end{eqnarray*} where equality holds when $k=n$.

 By the convexity of the function $f(t)=\log(1+e^{pt})$ for $p\ge 1$, we obtain from \cite[p. 156]{MOA11} that
 \begin{eqnarray*}\det(I_r+|X|^p)\cdot\det(I_{n-r}+|Z|^p)=\det\left(I_n+\begin{bmatrix} |X|^p &0 \\0 & |Z|^p\end{bmatrix}\right)\le \det(I_n+|T|^p).
\end{eqnarray*} This proves (\ref{thm3e}).

If $Y=0$, then clearly  \begin{eqnarray*}
\det\left(I_n+|T|^p\right)=\det\left(I_r+|X|^p\right)\cdot \det\left(I_{n-r}+|Z|^p\right).\end{eqnarray*}

Conversely, if  \begin{eqnarray*}
\det\left(I_n+|T|^p\right)=\det\left(I_r+|X|^p\right)\cdot \det\left(I_{n-r}+|Z|^p\right),\end{eqnarray*} the strict convexity of $f(t)=\log(1+e^{pt})$, $p\ge 1$, gives that  \begin{eqnarray*} \prod_{j=1}^k\lambda_j\left(\begin{bmatrix} |X| &U_1Y\\ 0 &  |Z|\end{bmatrix}\right)=\prod_{j=1}^k \sigma_j\left(\begin{bmatrix} |X| &U_1Y\\ 0  &  |Z|\end{bmatrix}\right), \qquad k=1, \ldots, n.
\end{eqnarray*}  And so $\begin{bmatrix} |X| &U_1Y\\ 0  &  |Z|\end{bmatrix}$ is normal, which is the case only if $U_1Y=0$ and therefore $Y=0$.
\end{proof}

Nevertheless, (\ref{e31}) does not have such an analogue. We show by an example that, in general, it is not true that
\begin{eqnarray}\label{e21} \det\left(|T_1|+|T_2|\right)\ge\det\left(|X_1|+|X_2|\right)\cdot\det\left(|Z_1|+|Z_2|\right),
\end{eqnarray} where $T_1, T_2$, are as in Theorem \ref{thm1}.

\begin{eg} Taking
\begin{eqnarray*}T_1=\left[\begin{array}{cc:cc} 2&-3&9&-1\\-4&15&1&-19 \\ \hdashline
 0&0&0&-2 \\ 0&0&-4&19 \end{array}\right], \qquad  T_2=\left[\begin{array}{cc:cc} 0&1&6&0\\4&-12&12&10\\ \hdashline
 0&0&14&-2 \\ 0&0&23&-3 \end{array}\right],
\end{eqnarray*} a simple calculation using Matlab gives \begin{eqnarray*}\det\left(|T_1|+|T_2|\right)= 5193.1<\det\left(|X_1|+|X_2|\right)\cdot\det\left(|Z_1|+|Z_2|\right)=20248.\end{eqnarray*} \end{eg}


\end{document}